\documentclass[a4paper]{lipics-v2021}

\newtheorem*{lemma*}{Lemma}
\nolinenumbers

\hideLIPIcs  

\title{The Shortest Temporal Exploration Problem} 

\author{Stefan {Balev}}{Université Le Havre Normandie, Univ Rouen Normandie, INSA Rouen Normandie, Normandie
Univ, LITIS UR 4108, F-76600 Le Havre, France}{stefan.balev@univ-lehavre.fr}{}{}
\author{Éric {Sanlaville}}{Université Le Havre Normandie, Univ Rouen Normandie, INSA Rouen Normandie, Normandie
Univ, LITIS UR 4108, F-76600 Le Havre, France}{eric.sanlaville@univ-lehavre.fr}{https://orcid.org/0000-0001-9482-3945}{}
\author{Antoine {Toullalan}}{Université Le Havre Normandie, Univ Rouen Normandie, INSA Rouen Normandie, Normandie
Univ, LITIS UR 4108, F-76600 Le Havre, France}{antoine.toullalan@univ-lehavre.fr}{https://orcid.org/0009-0000-1239-9424}{}
\authorrunning{S. Balev, É. Sanlaville and A. Toullalan} 

\Copyright{Stefan Balev, Éric Sanlaville and Antoine Toullalan} 



\category{} 

\relatedversion{} 

\funding{This work was Supported by the French ANR, project ANR-22-CE48-0001 (TEMPOGRAL)}

\acknowledgements{}

\EventEditors{}
\EventNoEds{2}

\begin{document}

\maketitle

\begin{abstract}
A temporal graph is a graph for which the edge set can change from one time step to the next. This paper considers undirected temporal graphs defined over $L$ time steps and connected at each time step.
We study the Shortest Temporal Exploration Problem (STEXP) that, given all the evolution of the graph, asks for a temporal walk that starts at a given vertex, moves over at most one edge at each time step, visits all the vertices, takes at most $L$ time steps and traverses the smallest number of edges.
.
We prove that every constantly connected temporal graph with $n$ vertices can be explored with $O(n^{1.5})$ edges traversed within $O(n^{3.5})$ time steps.
This result improves the upper bound of $O(n^2)$ edges for an exploration provided by the upper bound of time steps for an exploration which is also $O(n^2)$.
Morever, we study the case where the graph has a diameter bounded by a parameter $k$ at each time step and we prove that there exists an exploration which takes $O(kn^2)$ time steps and traverses $O(kn)$ edges.
Finally, the case where the underlying graph is a cycle is studied and tight bounds are provided on the number of edges traversed in the worst-case if $L\geq 2n-3$.
\end{abstract}

\section{Introduction}
Many real-life networks are not static objects because their connections may vary over time (e.g. transportation, communication, social networks,...). A tool to modelise and study these networks is temporal graphs whose edge set changes over time. More formally, a \emph{temporal graph }$\mathcal{G}$ of lifetime $L$ is usually defined as a sequence of undirected graphs $(G_1,G_2,...,G_L)$ called \emph{snapshots}. We refer to the surveys\cite{survey1,survey2} on the algorithmic aspects of temporal graphs for more details. 
The adaptation of path-related problems from static graphs to temporal graphs raises several questions: mainly, the notion of path, and the optimization criteria to use. This is shown by the study of some classical path-related problems transposed to temporal graphs, for example the shortest path prolem\cite{bui_xuan,spath}, the Steiner tree problem\cite{steiner} or the exploration problem\cite{erlebach_temp_explo,faster_explo}. In a static graph, a set of edges connecting two vertices $u$ and $v$ is called a \emph{path} between $u$ and $v$, but in a temporal graph the equivalent is a \emph{journey} from $u$ to $v$ which is a sequence of edges successive in time from $u$ to $v$. It means that, in temporal graphs, at least three criteria can be optimized: minimizing the \emph{number of edges}, minimizing the \emph{arrival time} and minimizing the \emph{duration}(the difference between the arrival time and the starting time of the agent that does the journey).
An extensive study of these criteria for the journeys in temporal graphs can be found in the article of Bui-Xuan et al.\cite{bui_xuan}.\\
A well studied path-related problem in temporal graphs is the exploration problem introduced by Michail and Spirakis\cite{TSP_in_temp_graph}. They make the assumption that the graph is connected at each time step because it guarantees that an exploration exists if the lifetime $L$ is large enough. 
We note the Temporal Exploration problem which minimizes the \emph{arrival time} the ETEXP (Earliest Temporal Exploration Problem). The problem that aims at minimizing the number of edges is the STEXP (Shortest Temporal Exploration Problem) and the problem that aims at minimizing the duration of the journey is the FTEXP (Fastest Temporal Exploration Problem).
ETEXP is the problem where an agent starts at a given vertex, aims at visiting all the vertices (traversing at most one edge per time step) and has the \emph{smallest arrival time}. Most papers studied ETEXP under the assumption that the graph is constantly connected \cite{erlebach_temp_explo,two_moves,ilcinkas_ring,ilcinkas_cactus,faster_explo,k_edge_deficient} and this paper also considers only constantly connected temporal graph so we simply refer to them as "temporal graphs". Results on ETEXP are summarized on table \ref{tab:tab1}. The static graph whose edge set is the union of the edge sets of each snapshot is denoted the \emph{underlying graph} and the total number of vertices is denoted by $n$.

\begin{table}[]
    \centering
    \begin{tabular}{|c|c|c|c|}
    \hline
     \textbf{Underlying graph} & \textbf{Lower bound} & \textbf{Upper Bound} & \textbf{References} \\
     \hline
     General case & $\Theta(n^2)$ & $\Theta(n^2)$ & \cite{erlebach_temp_explo}\\ \hline
    Cycle & $2n - 3$ & $2n -3$ & \cite{ilcinkas_ring}\\ \hline
    Treewidth k &  &$O(n^{1.5}k\log n)$  & \cite{faster_explo}\\ \hline
    Planar  & $\Omega(n\log n)$ & $O(n^{1.75}\log n)$& \cite{erlebach_temp_explo} and \cite{faster_explo}\\ \hline
    Grid of size $2\times m$ & & $O(m\log^3 m)$ & \cite{erlebach_temp_explo}\\ \hline
    Cactus &  & $O(n\frac{\log n}{\log\log n})$ & \cite{ilcinkas_cactus}\\ \hline
    Cycle with k chords  &  & $6kn$ & \cite{faster_explo}\\ \hline
    \hline
    \textbf{Other restricted cases of temporal graph} & \textbf{Lower bound} & \textbf{Upper Bound} & \textbf{References} \\ \hline
    "k-edge-deficient" graph & $\Omega(n\log k)$ & $O(kn\log n)$ & \cite{k_edge_deficient}\\ \hline
    Each snapshot has a bounded maximum degree & $\Omega(n\log n)$ & $O(n^{1.75})$ & \cite{two_moves} and \cite{erlebach_temp_explo}\\ \hline
    \end{tabular}
    \caption{Main results for ETEXP. A temporal graph is "k-edge-deficient" if each snapshot has at most $k$ edges missing comparing to the underlying graph. A cactus is a graph where two cycles share at most one vertex.}
    \label{tab:tab1}
\end{table}
While much work has been done on the study of the ETEXP (often called TEXP), to the best of our knowledge no specific studies exist on the STEXP even though Michail and Spirakis\cite{TSP_in_temp_graph} studied the $TTSP(1,2)$ problem where at each time step the snapshot is a complete graph with a cost of $1$ or $2$ on each edges. The decision problem of the $TTSP(1,2)$ with cost $n$ may be associated to the decision problem of STEXP variant where the number of traversed edges is exactly $n-1$. This problem is NP-complete since the goal is to compute a temporal hamiltonian cycle of total cost $n$ in the temporal graph.
\newline
\newline
\textbf{\indent Our contributions}
\newline
\newline
This paper focuses on finding upper bounds for STEXP for different classes of temporal graphs.
Results are presented for the STEXP on temporal graphs with the assumption that the agent knows all the evolution of the graph and can traverse at most one edge per time step, the class of journeys used and more definitions are presented in section \ref{sec:preliminaries}. Section \ref{sec:general} contains the proof that there is a value $L_l$ which is $O(n^{3.5})$ such that for any temporal graph $\mathcal{G}$ with a lifetime $L\geq L_l$, there is an exploration traversing $O(n^{1.5})$ edges. This result improves the upper bound of $O(n^2)$ edges traversed for an exploration that is provided by the existence of an exploration that takes $O(n^2)$ time steps\cite{TSP_in_temp_graph}. 
Then, section \ref{subsec:diam} contains the proof that if each snapshot of the graph has a bounded diameter $k$ and if the lifetime is $L\geq kn^2$, an exploration exists that traverses $O(kn)$ edges. Erlebach et. al\cite{erlebach_temp_explo} have proven that there exists a family of temporal graphs with snapshots of bounded diameter (each snapshot is a star) for which all explorations take $\Omega(n^2)$ time steps and traverse $\Omega(n)$ edges, and the later result implies that for every temporal graph of bounded diameter, there is an exploration that takes $O(n^2)$ time steps and traverses $O(n)$ edges.
Finally section \ref{subsec:cycle} studies the case where the underlying graph is a cycle if $L\geq 2n-3$ because Ilcinkas et al.\cite{ilcinkas_ring} proved that an exploration exists if $L\geq 2n-3$. We prove that the worst-case number of edges traversed is $2n-3$ if $L=2n-3$, but if $L>2n-3$ there always exists an exploration traversing at most $\lfloor\frac{3}{2}(n-1) \rfloor$ edges.  
\section{Preliminaries}
\label{sec:preliminaries}
\begin{definition}{Temporal graph}
\\A temporal graph $\mathcal{G}$ with a vertex set $V_{\mathcal{G}}$ and a lifetime $L_{\mathcal{G}}$ is a sequence of static graphs $(G_1,G_2,...G_{L_{\mathcal{G}}})$ where $G_i=(V_{\mathcal{G}},E_i)$ is called the snapshot at the time step $i\in[1,L_{\mathcal{G}}]$.The underlying graph of $\mathcal{G}$ is $G=(V_{\mathcal{G}},E)$ with $V_{\mathcal{G}}$ the vertex set of $\mathcal{G}$ and $E$ the union of the edge sets of all the snapshot of $\mathcal{G}$.  
\end{definition}
In the rest of the paper, if no ambiguity arises, we will note the lifetime by $L$ and the vertex set by $V$. We restrict our attention to always-connected graph. Morever, when we talk about journeys, we often refer to an agent to describe the edges that are used on the temporal graph.
\begin{definition}{Path}
    \\Let $\mathcal{G}$ be a temporal graph. A path between the vertices $u\in V$ and $v\in V$ is a set of edges \emph{all belonging to the same snapshot} and connecting $u$ and $v$.
\end{definition}
We highlight the fact that a path is associated to exactly one snapshot which is different froma the definiton of a journey presented below.
\begin{definition}{Journey}
    \\Let $\mathcal{G}$ be a temporal graph. A journey from $u\in V$ to $v\in V$ is a sequence of edges of increasing time steps (but not necessarily consecutive), $[((v_0,v_1),t_1),((v_1,v_2),t_2),...,(v_{l-1},v_l),t_l)]$, where $(v_i,v_{i+1})\in E_i$, $v_0=u$ and $v_l=v$.So an agent can go from $u$ to $v$ by traversing the edges of the sequence at the time step indicated.\\
    There exist two types of journeys:
    \begin{itemize}
        \item the non-strict journeys where the agent can traverse several edges per time step. 
        \item the strict journeys where the agent can traverse at most one edge per time step, i.e. there is $\forall{i<l}, t_i< t_{i+1}$.
    \end{itemize}
\end{definition}
Since constantly connected temporal graphs are studied, if non-strict journeys were allowed, there would be an exploration in $1$ time step with at most $2n-3$ edge traversals. So only strict journeys are considered which means that at each time step the agent can either stay on a vertex or traverse an edge.  
\begin{definition}{Journey parameters}
    \begin{itemize}
        \item{Journey length}
        \\Let $\mathcal{G}$ be a temporal graph and the journey $J=((e_1,t_1),(e_2,t_2),...,(e_k,t_k))$ in $\mathcal{G}$.
        The length of $J$ is the number of edges of the journey, i.e. $k$.
        \item{Arrival time}
        \\The arrival time is the time step where the agent traverse the last edge which is $t_k$.
        \item{Journey time}
        \\The journey time is the number of time steps of the journey which is $t_k-t_1+1$.
    \end{itemize}
\end{definition}

The notation $u\rightsquigarrow_r v$ is used for a journey from $u$ to $v$ in $\mathcal{G}$ whose length is less than or equal to $r$ (we remark that $r$ is not necessarily an integer). Similarly, for a snapshot $G_i$ of $\mathcal{G}$ we note $u\leftrightarrow_r v$ a path connecting $u$ and $v$ with a number of edges less than or equal to $r$ in $G_i$. 
\begin{lemma}{Reachability\cite{TSP_in_temp_graph}}
    \label{lemma:l1}
    \\Let $\mathcal{G}$ be a temporal graph of lifetime $L\geq n-1$ and $(u,v)$ a pair of vertices. Then $\forall t\leq L-n+2$ there is a strict journey from u to v starting at t, whose journey time is at most $n-1$.
\end{lemma}
\section{The STEXP in the general case}
\label{sec:general}
In this section the following notation and definition will be used.
Let $\mathcal{G}_{[a,b]}$ be the temporal graph corresponding to $\mathcal{G}$ between the time steps $t=a$ and $t=b$.
\begin{definition}{k-journey-free set\\}
Let $\mathcal{G}$ be a temporal graph, $S$ a subset of $V$ and $1< k<n$. We define $X_k^S$ a k-journey-free set as a subset of $S$ such that there is no pair of distinct vertices $u,v\in X_k^S$ such that both of the journeys $u \rightsquigarrow_k v$ or $v \rightsquigarrow_k u$ exist and $X_k^S$ has a \emph{maximal size}.\\
We denote $X_k$ a k-journey-free set $X_k^V$.
\end{definition}
Note that it is possible that $|X_k|=1$, i.e. $X_k=\{v\}$, in which case for any vertex $u\neq v$ the journeys $u \rightsquigarrow_k v$ and $v\rightsquigarrow_k u$ exist.
Morever, we remark that $X_k^S$ can correspond to different (even disjoint) subsets of $S$ as shown in figure \ref{fig:img_ex_Xk}. In this figure, a temporal graph is defined on 3 time steps and two exemples of $X_2$ sets in the graph of lifetime $3$ : $\{A,D\}$ and $\{E,C\}$. For the set $\{A,D\}$ there is no journeys of $2$ edges or less from $A$ to $D$ whose last edge is traversed at most at $3$ so $A$ and $D$ can be in the same set $X_2$ but there is a journey of $2$ edges from $A$ to $C$ and from $C$ to $A$ in the three time steps so $A$ and $C$ can't be in the same set $X_2$. Finally, if $S=\{A,B,D\}$, then $X_k^S=\{A,D\}$.
\begin{figure}[!h]
    \centering
    \includegraphics[scale=0.25]{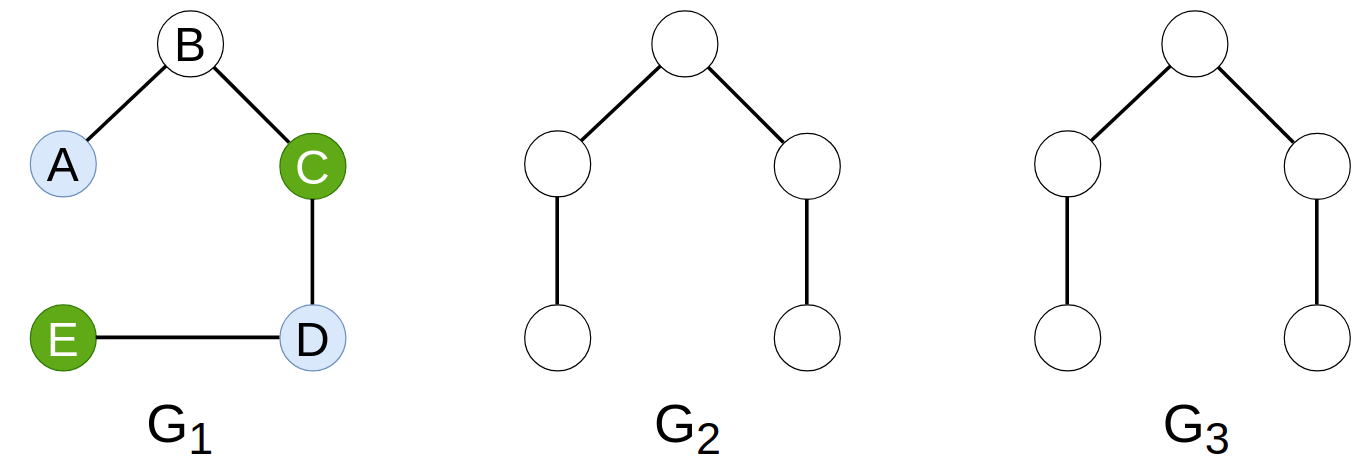}
    \caption{A representation of a temporal graph with $L=3$. Two examples of $X_2$ sets are presented: the subsets of vertices colored blue ($A$ and $D$) and green ($E$ and $C$).}
    \label{fig:img_ex_Xk}
\end{figure}

Now we will prove that there exists an exploration traversing $O(n^{1.5})$ edges that takes $O(n^{3.5})$ time steps. The proof is split into a number of lemmas, whose proofs are left in the appendix. Recall that the assumption is that each snapshot is connected and the agent already knows all the evolution of the graph.\\
Let $(u,v)$ be a pair of vertices and $k,1\leq k<n$. Our first step is proving a sufficient condition for the existence of a journey of at most $k$ edges from $u$ to $v$.

\begin{lemma}
   Let $\mathcal{G}$ be a temporal graph of $n$ vertices and $k$ such that $1\leq k< n$, the lifetime of the graph is $L\geq kn$. \\ Let $u,v$ be two distinct vertices of $V$, such that there exist $kn$ distinct paths (i.e on distinct snapshots) with a number of edges less than or equal to $k$ connecting $u$ and $v$, each of these paths $u\leftrightarrow_k v$ is associated with exactly one time step. Then there exists a journey $u \rightsquigarrow_k v$ in $\mathcal{G}$ using only these time steps.
    \label{lemma:l4}
\end{lemma}
\begin{proof}
    See the proof \ref{proof:p1}.
\end{proof}

\begin{lemma}
    Let $G=(V,E)$ be a connected static graph of $n$ vertices and $k,\mbox{ such that } 1< k<n$. Let $Z$ be a subset of vertices of maximal size such that $\nexists{(u,v)\in Z},u \leftrightarrow_k v$, then we have $|Z|<2n/k$.
    \label{lemma:l5}
\end{lemma}
\begin{proof}
    See the proof \ref{proof:p2}.
\end{proof}

We remark that, given a static graph $G$ of vertex set $V$, a real $k, 1< k<n$, and a set $Z$ of maximal size that verifies $\nexists{(u,v)\in Z},u \leftrightarrow_k v$, we have $\forall{u\in V\backslash Z}, \exists{v\in Z}$ such as there is a path $u \leftrightarrow_k v$. This is because if there weren't such path of size $k$, a vertex $u$ could be added to the set $Z$ which is a contradiction because $Z$ has a maximal size.\\
The two lemmas \ref{lemma:l4} and \ref{lemma:l5} are used to prove Lemma \ref{lemma:l6}. The Lemmastates that for a certain value of $k$, if $L$ is great enough, the size of $k$-journey-free sets have an explicit upper bound.
\begin{lemma}
Let $\mathcal{G}$ be a temporal graph, let $q\in\mathbb{N}^*$, let $S$ be a subset of vertices such that $|S|\leq n/q$ and let $k=2\sqrt{nq}$.
 If the lifetime $L$ verifies $L\geq 4\frac{n^{2.5}}{\sqrt{q}}$ then for every set $X_k^S$ we have $|X_k^S|< \sqrt{n/q}$.
    \label{lemma:l6}
\end{lemma}
\begin{proof}
    See the proof \ref{proof:p3}.
\end{proof}
\begin{figure}[!h]
        \centering
        \includegraphics[scale=0.21]{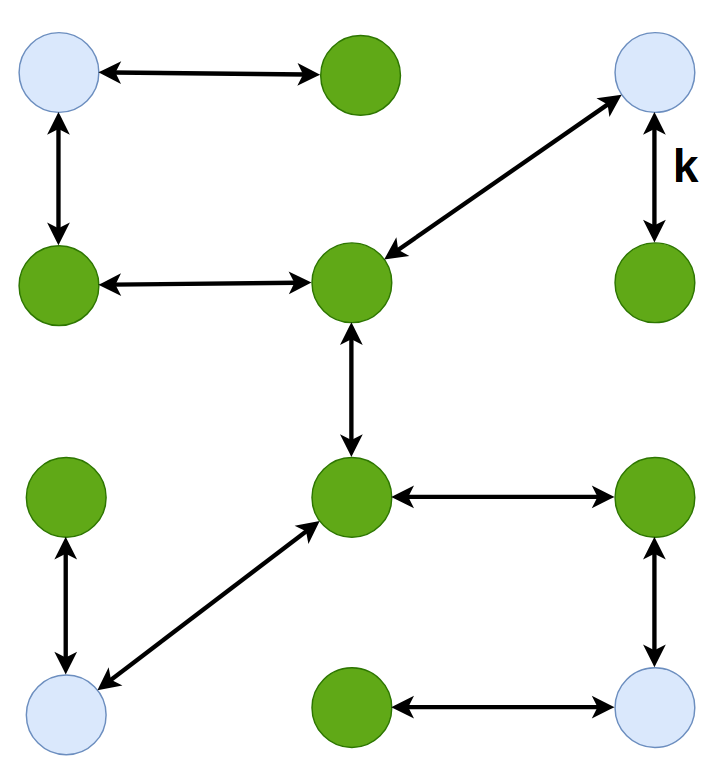}
        \caption{Representation of a subset of vertices $S$ in the graph $\mathcal{G}_{[t_1,t_2]}$ with $t_1+4\frac{n^{2.5}}{\sqrt{q}}\leq t_2$. We represent in blue the vertices of $X_k^S$, with $|X_k^S|<\sqrt{n/q}$, and in green the vertices of $S\backslash X_k^S$. Each arrow represents a journey of less than or equal to $k$ edges in $\mathcal{G}_{[t_1,t_2]}$. Note that each vertex of $S$ is either a vertex of $X_k^S$ or connected by a journey of $k$ edges or less (in either direction) to a vertex of $X_k^S$. This figure illustrates Lemma \ref{lemma:l6}.}
        \label{fig:imgsupp}
    \end{figure}
Lemma \ref{lemma:l6} shows that for any subset of vertices $S$ there is a subset of $S$ (i.e $X_k^S$) that covers all the vertices $S$ with journeys of $k$ edges such as presented in figure \ref{fig:imgsupp}. More formally, given a temporal graph $\mathcal{G}$ of vertex set $V$, a real $1< k<n$, and a set $X_k$, we have $\forall{u\in V\backslash X_k}, \exists{v\in X_k}$ such as there is a journey $u \rightsquigarrow_k v$ and a journey $v \rightsquigarrow_k u$. This is because if there wasn't such journeys of size $k$, a vertex $u$ could be added to the set $X_k$ which is a contradiction because $X_k$ has a maximal size. This property is illustrated by figure \ref{fig:imgsupp}.\\
This property is used in the proof of Lemma \ref{lemma:l7} to create a journey that, given a subset of vertices $S$, visits a fraction of the vertices of $S$. This journey is a concatenation of journeys of at most $k$ edges between the different vertices of $S$ visited with $k$ defined in the previous Lemma \ref{lemma:l6}. The construction of this journey is illustrated by figure \ref{fig:img3}.
\begin{lemma}
    Let $\mathcal{G}$ be a temporal graph, let $q\in\mathbb{N}^*$ and $S$ a subset of vertices such that $|S|\leq n/q$.
    If $L\geq 5\frac{n^3}{q}$, then there exists a journey in $\mathcal{G}$ visiting $|S|\sqrt{q/n}$ vertices of $S$ by traversing at most $2n$ edges.
    \label{lemma:l7}
\end{lemma}
\begin{proof}
    See the proof \ref{proof:p4}.
\end{proof}
\begin{figure}[!h]
    \centering
    \includegraphics[scale=0.2]{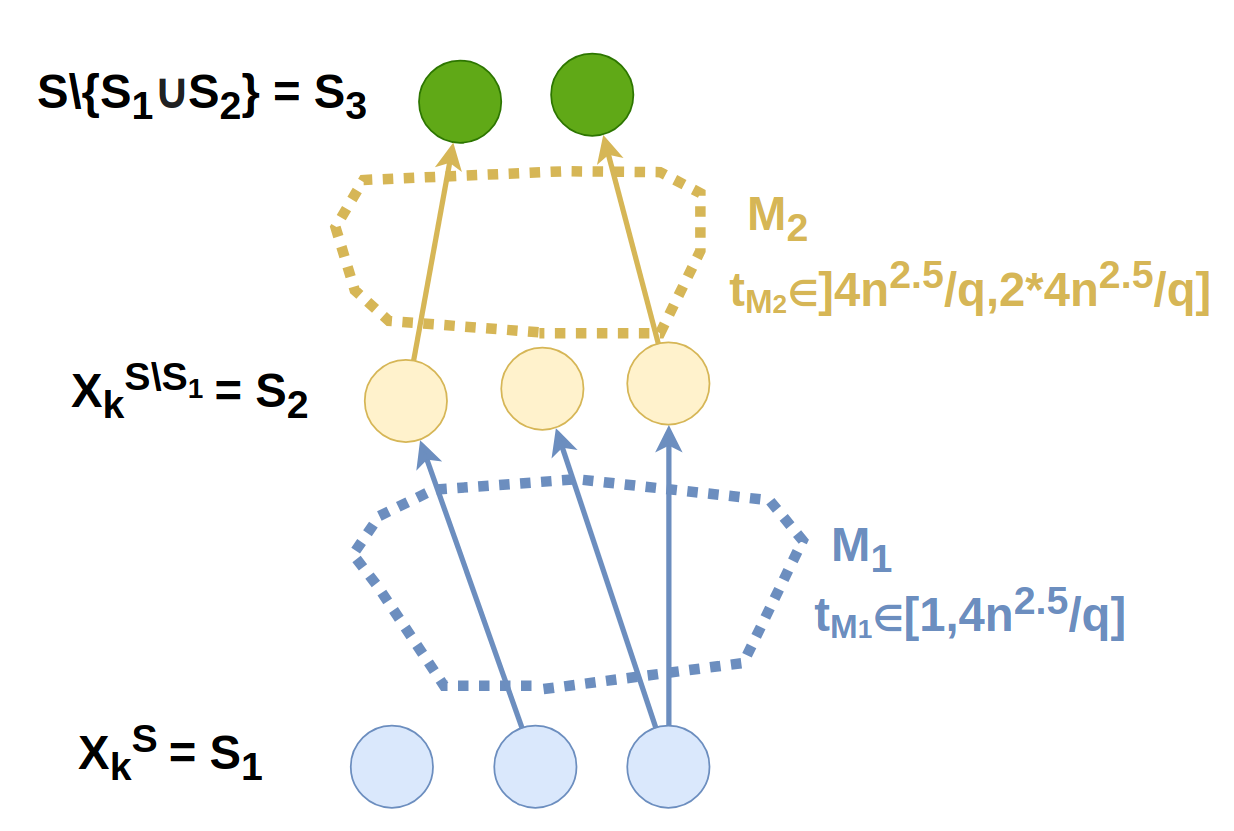}
    \caption{A representation of the set S partitioned into three subsets $\{S_1,S_2,S_3\}$ which are connected by journeys of the set $M=\{M_1,M_2\}$ of lengths at most $k$. We also represent the time windows associated with the journeys of $M_1$ and $M_2$ (see the proof \ref{proof:p4} for definitions of $S_i$ and $M_i$).This figure illustrates the proof of Lemma \ref{lemma:l7}.}
    \label{fig:img3}
\end{figure}
Lemma \ref{lemma:l7} allows us to prove Lemma \ref{lemma:l8} which states that, given any subset of vertices $S$, a journey traversing $O(n^{1.5})$ edges can visit half of the vertices of $S$ if $L$ is great enough. This journey is a concatenation of journeys such as represented on figure \ref{fig:img4}.
\begin{lemma}
     Let $\mathcal{G}$ be a temporal graph with a starting vertex $s$, let $q\in\mathbb{N}^*$ with $q<n/4$ and $S$ a subset of vertices such that $|S|= n/q$.
     If the lifetime of $\mathcal{G}$ is $L\geq 6\frac{n^{3.5}}{q^{1.5}}$ then there exists a journey, starting at $s$, visiting $|S|/2$ vertices of $S$ and traversing at most $4\frac{n^{1.5}}{\sqrt{q}}$ edges. 
     \label{lemma:l8}
\end{lemma}
\begin{proof}
    See the proof \ref{proof:p5}.
\end{proof}

\begin{figure}[!h]
        \centering
        \includegraphics[scale=0.2]{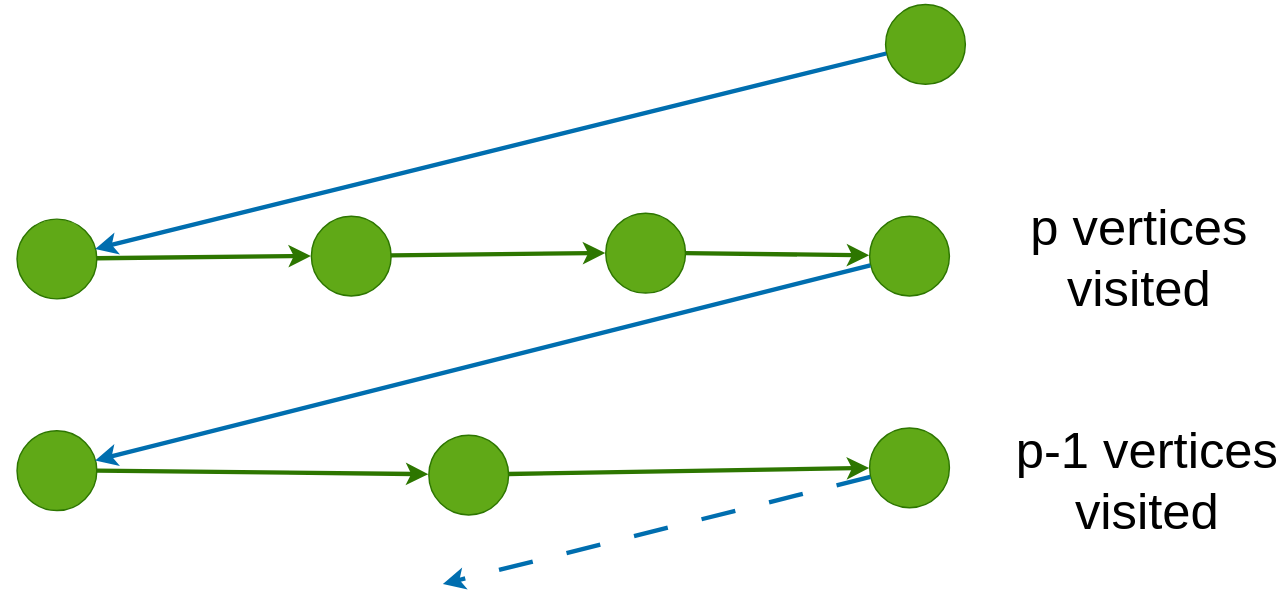}
        \caption{A representation of the concatenated journeys forming the journey described in Lemma \ref{lemma:l8} where the vertices of $S$ are represented in green. A blue arrow is a journey traversing at most $n-1$ edges such as defined in Lemma \ref{lemma:l1} (journey of type $2$). And a set of green arrows that follow each other is a journey of type $1$ such as defined in Lemma \ref{lemma:l7} that traverses at most $2n$ edges and visit between $p$ and $1$ vertices of $S$(see the proof \ref{proof:p5} for a definition of $p$). In total these concatenated journeys visit $|S|/2$ vertices of the set $S$.}
        \label{fig:img4}
    \end{figure}
The theorem \ref{theorem:t1} is a corollary of Lemma \ref{lemma:l8}.
\begin{theorem}
\label{theorem:t1}
     Let $\mathcal{G}$ be a temporal graph with a lifetime $L\geq 12n^{3.5} + 4n$.
     Then there exists an exploration of $\mathcal{G}$ traversing $O(n^{1.5})$ edges.
\end{theorem}
\begin{proof}
    Let $S$ be the set of unvisited vertices. Initially we set $S=V$ (to simplify our proof, we ignore the fact that the starting vertex $s$ has already been visited). To explore $\mathcal{G}$, we traverse a journey visiting $|S|/2$ vertices of the set $S$ as defined in Lemma \ref{lemma:l8} as long as we have $|S|>4$. Initially we have $|S|=n$, then $|S|=n/2$, then $|S|=n/4$, etc. So we traverse a journey as defined in Lemma \ref{lemma:l8} at most $\log_2n-2$ times to explore in total at least $n-4$ distinct vertices of $\mathcal{G}$. After these journeys we can visit the remaining $4$ unvisited vertices of $\mathcal{G}$ by making $4$ journeys of $n-1$ time steps and $n-1$ edges.\\
    To compute the number of time steps and edges in the algorithm, we first show that $\sum_{i=0}^{\log_2(n)-2}\frac{1}{\sqrt{2^i}}<4$. Indeed $\sum_{i=0}^{\log_2(n)-2}\frac{1}{\sqrt{2^i}}<\sum_{i=0}^{\infty}\frac{1}{\sqrt{2^i}}<2+2/\sqrt{2}<4$.\\

    We compute the number of edges traversed in the $\log_2n-2$ journeys that each visit half of the remaining unvisited vertices: $\sum_{i=0}^{\log_2(n)-2}4\frac{n^{1. 5}}{\sqrt{2^i}} =4n^{1.5}\sum_{i=0}^{\log_2(n)-2}\frac{1}{\sqrt{2^i}}<16n^{1.5}$. 
    
    So the journey of the exploration algorithm traverses less than $16n^{1.5} + 4(n-1)$ edges, i.e. $O(n^{1.5})$ edges.\\
    In the same way, we compute the number of time steps required to traverse these $\log_2(n)-2$ journeys each visit half of the remaining unvisited vertices: $\sum_{i=0}^{\log_2(n)-2}6\frac{n^{3. 5}}{2^i\sqrt{2^i}} < 6n^{3.5}\sum_{i=0}^{\log_2(n)-2}\frac{1}{2^i} < 6n^{3.5}\times2=12n^{3.5}$.
    So the complete journey of the exploration algorithm requires less than $12n^{3.5} + 4(n-1)<L$ time steps.
    So there is an exploration of $\mathcal{G}$ in $O(n^{3.5})$ time steps traversing $O(n^{1.5})$ edges.
\end{proof}

\section{The STEXP for restricted classes of temporal graphs}
\label{sec:restricted}
The result for the general case presents an upper bound of $O(n^{1.5})$ edges for an exploration with $O(n^{3.5})$ time steps, and we have a lower bound of $\Omega(n)$ edges traversed for temporal graphs with the same lifetime. So our upper bound on the number of edges does not close the gap with the lower bound in the context of STEXP. This motivates us to study more restricted classes of temporal graphs. First we study temporal graphs with a bounded diameter at each time step and we prove that we have an upper bound that close the gap with the lower bound for the number of edges traversed if the temporal graph has a lifetime in $O(n^2)$. Next we study the case where the underlying graph is a cycle, we also prove tight bounds for the number of edges if $L\geq 2n-3$ and we present an interesting result where the worst-case scenario for the number of edges traversed is different for $L=2n-3$ and $L> 2n-3$. 

\subsection{The case of bounded-diameter temporal graphs}
\label{subsec:diam}
\begin{theorem}
    Let $\mathcal{G}=(V,E_t),t\leq L$ be a temporal graph of $n$ vertices and $0< k< n$ an integer, the lifetime of the graph is $L\geq kn^2$.
    If $\forall t\leq L$, the snapshot $G_t$ has a diameter less than or equal to $k$, then there exists an exploration of $\mathcal{G}$ traversing less than $kn$ edges.
\end{theorem}
\begin{proof}
    By Lemma \ref{lemma:l4}, given two vertices $u,v$, if there are $kn$ distincts snapshots that include a path $u\leftrightarrow_k v$ then there is a journey $u\rightsquigarrow_k v$. We remark that in a snapshot $G_i$ of diameter less than or equal to $k$, $\forall (u,v)\in V^2, \exists u\leftrightarrow_k v$. So in every time window of size $kn$, by assumption we have $\forall (u,v)\in V^2, \exists u\rightsquigarrow_k v$. Therefore we  construct an exploration of the temporal graph by chosing any permutation $P$ of the vertices beginning by the starting vertex $u_1$: $P=[u_1,u_2,...,u_n]$ and $\forall i, 0<i<n$ the agent makes a journey $u_i\rightsquigarrow_k u_{i+1}$ in the time window $[1+(i-1)kn,1+ikn[$ of size $kn$. So we create an exploration of $\mathcal{G}$ in $kn(n-1)< L$ time steps by traversing at most $k(n-1)<kn$ edges.
\end{proof}
Erlebach et al.\cite{erlebach_temp_explo} have presented a family of temporal graphs such that each snapshot has a diameter $2$ and all explorations take $\Omega(n^2)$ time steps with $\Omega(n)$ edges traversed. We have proven that, for each temporal graph with $L=\Omega(n^2)$, there is an exploration in $O(n^2)$ time steps and $O(n)$ edges traversed if each snapshot of $\mathcal{G}$ has a bouded diameter. Hence these bounds are tight. 
\subsection{The case of the cycle as the underlying graph}
\label{subsec:cycle}
In this section the following notations are used: $\mathcal{C}$ is a temporal cycle, i.e a temporal graph  whose underlying graph is a cycle, denoted $C$ and  The starting vertex of the exploration is a given vertex $s$. Morever we denote $m$ the edge of $C$ such that the number of edges from $s$ to the left extremity of $m$ is exactly $\lfloor \frac{n-1}{2} \rfloor$ and the number of edges from $s$ to the right extremity of $m$ is exactly $\lceil \frac{n-1}{2} \rceil$.
\begin{lemma}{cycle exploration time\cite{ilcinkas_ring}}
    \label{lemma:l2}
    \\Let $\mathcal{C}$ be a temporal cycle of lifetime $L\geq2n-3$ and $s$ the starting vertex, then there is an exploration of $\mathcal{C}$ starting at $s$.
\end{lemma}
\begin{lemma}{cycle property for the exploration\cite{erlebach_temp_explo}}
    \label{lemma:l3}
    \\Let $\mathcal{C}$ be a temporal cycle of lifetime $L\geq n-1$. Then for every $t$ such as $t\leq L-n+2$ there is a vertex called $s_h(t)$ (resp. $s_a(t)$) such that an agent, moving constantly clockwise (resp. counter-clockwise) on the cycle starting at $s_h(t)$ (resp. $s_a(t)$) at the time step $t$, explores $\mathcal{C}$ in $n-1$ time steps without being blocked by any absent edges.
\end{lemma}
\subsubsection{The upper and lower bounds on the number of edges for an exploration in $2n-2$ time steps}
\begin{lemma}
    Let $\mathcal{C}$ be a temporal cycle with a lifetime $L\geq 2n-2$ and a starting vertex $s$. Then there is an exploration of the graph starting from $s$ and traversing at most $\lfloor \frac{3}{2}(n-1) \rfloor$ edges.
    \label{lemma:l1_cycle}
\end{lemma}
\begin{proof}

By Lemma \ref{lemma:l3}, we know that at each time step $t$, there exist two vertices $s_h(t)$ and $s_a(t)$ of the $n$-vertices cycle such that if the agent is at the vertex $s_h(t)$,respectively $s_a(t)$, at the time step $t$ and always moves clockwise, resp. counter-clockwise, the agent traverses one edge per time step for $n-1$ time steps, thus visiting all the vertices in the cycle in $n-1$ time steps. We can therefore construct an exploration of the cycle in $2n-2$ time steps as follows: the agent traverses the journey from the starting vertex $s$ to $s_h(n)$ (or $s_a(n)$) in $n-1$ time steps according to Lemma \ref{lemma:l1} and then visits all the vertices of the cycle in $n-1$ time steps. In the rest of the proof we use the following notation, $s_h = s_h(n)$ and $s_a = s_a(n)$.
We consider two algorithms of exploration for the cycle:\\
\noindent\textbf{Algorithm 1:}
\begin{enumerate}
    \item The agent is on $s$ at timestesp $1$.
    \item The agent reaches $s_h$ in $n-1$ time steps going either clockwise or counter-clockwise.
    \item The agent moves always clockwise from the vertex $s_h$ until all the vertices are visited (the agent traverses at most $n-1$ edges). The exploration takes at most $n-1+n-1=2n-2$ time steps and the last edge is traversed at most at the time step $L$.
\end{enumerate}
\textbf{Algorithm 2:}
\begin{enumerate}
    \item The agent is on $s$ at time step $1$.
    \item The agent reaches $s_a$ in $n-1$ time steps going either clockwise or counter-clockwise (if both journeys take $n-1$ time steps, the counter-clockwise journey is chosen).
    \item The agent moves always counter-clockwise from the vertex $s_a$ until all the vertices are visited.(the agent traverses at most $n-1$ edges). The exploration takes at most $n-1+n-1=2n-2$ time steps and the last edge is traversed at most at the time step $L$..
\end{enumerate}

We then present an algorithm for exploring the cycle in at most $2n-2$ time steps and traversing at most $\lfloor \frac{3}{2}(n-1) \rfloor$ edges.
This algorithm consists of returning the exploration of the algorithm $1$ if it traverses at most $\lfloor \frac{3}{2}(n-1) \rfloor$ edges. If this is not the case, we return the solution of the algorithm $2$.
We prove that this algorithm always returns an exploration traversing at most $\lfloor \frac{3}{2}(n-1) \rfloor$ edges in $2n-2$ time steps.
First we study the result of the algorithm $1$: this algorithm returns a solution that can be divided into three cases (shown in figures \ref{fig:1_1},\ref{fig:1_2},\ref{fig:1_3}).\\
\begin{itemize}
\item{Case 1.1 : } In the first case shown in figure \ref{fig:1_1}, the journey from $s$ to the vertex $s_h$ is clockwise, so we can see that the total number of edges traversed is $n-1$. Our algorithm will return this solution ($n-1 < \lfloor \frac{3}{2}(n-1) \rfloor$).

\item{Case 1.2 : } In the second case shown in figure \ref{fig:1_2}, the journey from $s$ to $s_h$ is counter-clockwise, and $s_h$ is on the left size of $m$. Therefore the journey from $s$ to $s_h$ traverses at most $\lfloor \frac{n-1}{2} \rfloor$ edges, so the total number of edges traversed is at most $\lfloor \frac{3}{2}(n-1) \rfloor$ edges. Also our algorithm will return this solution.

\begin{figure}
    \begin{minipage}{0.45\textwidth}
        \centering
        \includegraphics[width=0.7\textwidth]{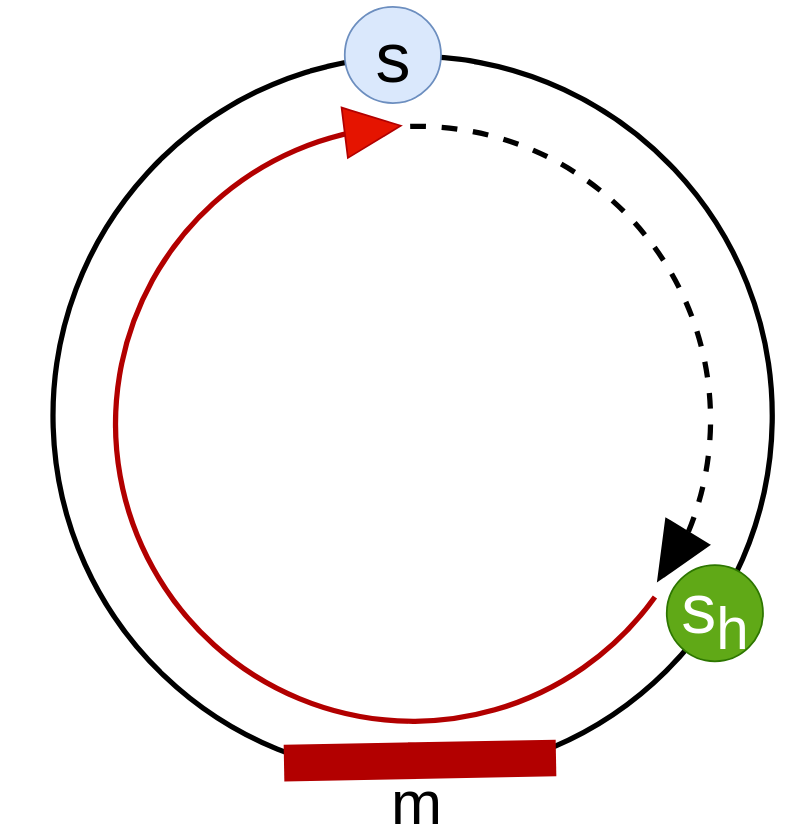} 
        \caption{The journey from $s$ to $s_h$ (dashed line) is clockwise from s.}
        \label{fig:1_1}
    \end{minipage}\hfill
    \begin{minipage}{0.45\textwidth}
        \centering
        \includegraphics[width=0.71\textwidth]{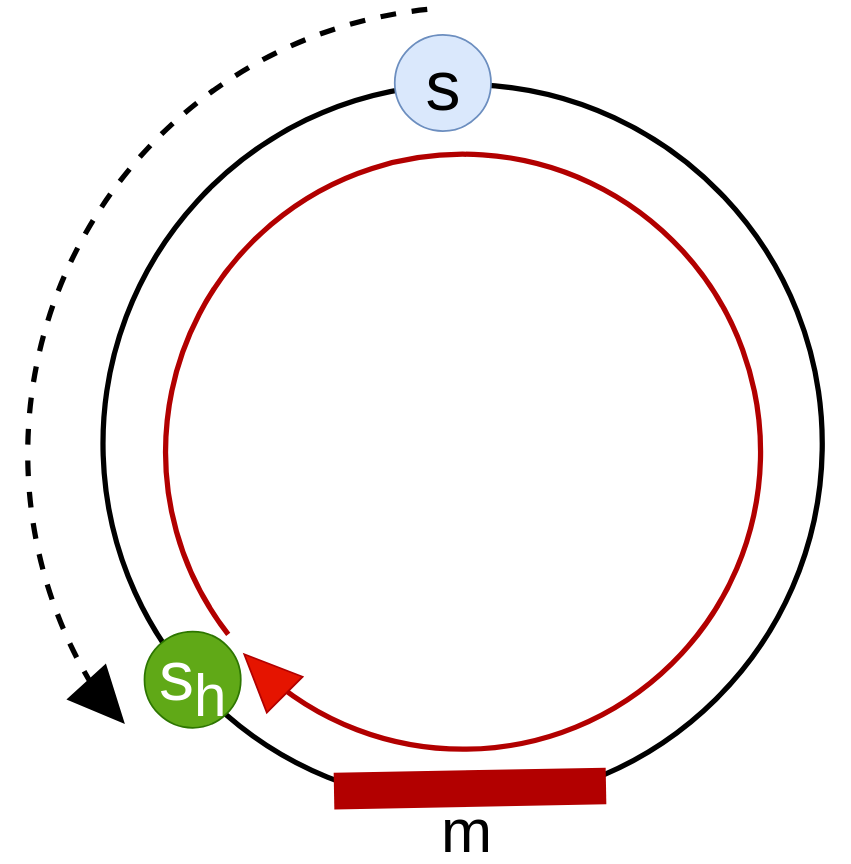} 
        \caption{The vertex $s_h$ is on the left side of $m$ and the journey from $s$ to $s_h$ (dashed line) is counter-clockwise from s.}
        \label{fig:1_2}
    \end{minipage}
\end{figure}

\item{Case 1.3 : } In the third case shown in figure \ref{fig:1_3}, the journey from $s$ to the vertex $s_h$ is counter-clockwise, and $s_h$ is on the \textbf{right} side of the edge $m$. So the journey from $s$ to $s_h$ traverses strictly more than $\lfloor \frac{n-1}{2} \rfloor$ edges. This gives us a total number of edges traversed that is strictly greater than $\lfloor \frac{3}{2}(n-1) \rfloor$ edges. In this case, the solution of algorithm $2$ will be returned.
\begin{figure}[!h]
\centering
\includegraphics[width=0.3\textwidth]{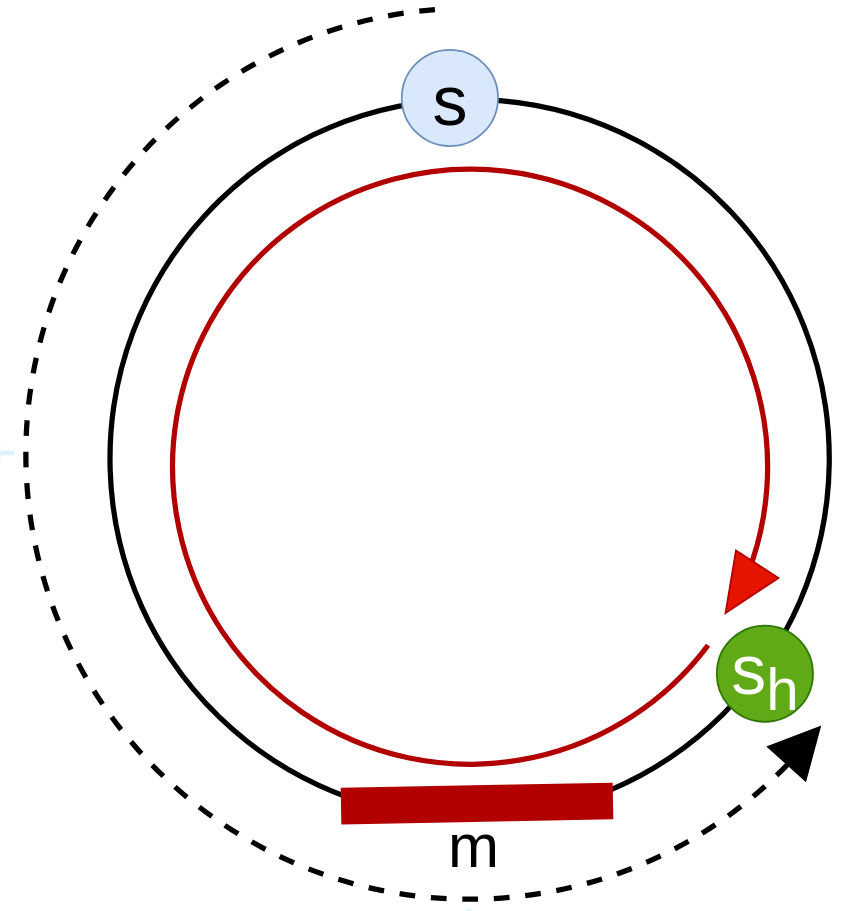}
\caption{The vertex $s_h$ is on the right side of $m$ and the journey from $s$ to $s_h$ (dashed line) is counter-clockwise from $s$.}
\label{fig:1_3}
\end{figure}
\end{itemize}
In the third case, the result of the algorithm $2$ is returned, otherwise the result of the algorithm $1$ is returned. The result returned by the algorithm $2$ can be divided into the same three distinct cases as algorithm $1$.
And similarly to the algorithm $1$, there is only one of the three cases that can return an exploration traversing more than $\lfloor \frac{3}{2}(n-1) \rfloor$ edges, which is represented in figure \ref{fig:2_3b}, we call this case 2.3. We prove that this case induces the existence of an exploration traversing less than  $\lfloor \frac{3}{2}(n-1) \rfloor$ edges.
The case 2.3 is when $s_a$ is on the left side of $m$ or $s_a$ is the right extremity of $m$ and the agent traverses a counter-clockwise journey from $s$ to $s_a$. Then the journey from $s$ to $s_a$ traverses at least $\lceil \frac{n-1}{2} \rceil$ edges, so the total number of edges traversed is greater than or equal to $\lfloor \frac{3}{2}(n-1) \rfloor$ (we have $\lceil \frac{n-1}{2} \rceil+n-1\geq\lfloor \frac{3}{2}(n-1) \rfloor$).

\begin{figure}[!h]
    \begin{minipage}{0.45\textwidth}
        \includegraphics[scale=0.15]{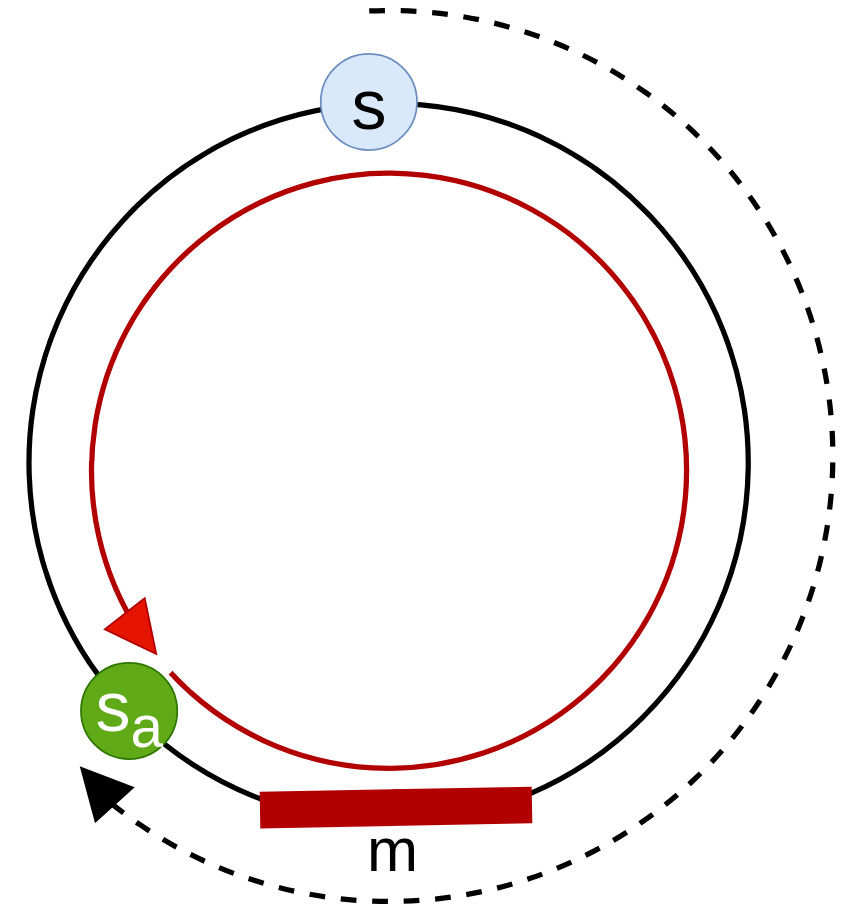} 
        \caption{This is the case 2.3. The vertex $s_a$ is on the left side of $m$ and there is a journey from $s$ to $s_a$ that goes clockwise.}
        \label{fig:2_3b}
    \end{minipage}
    \begin{minipage}{0.45\textwidth}
       \includegraphics[scale=0.15]{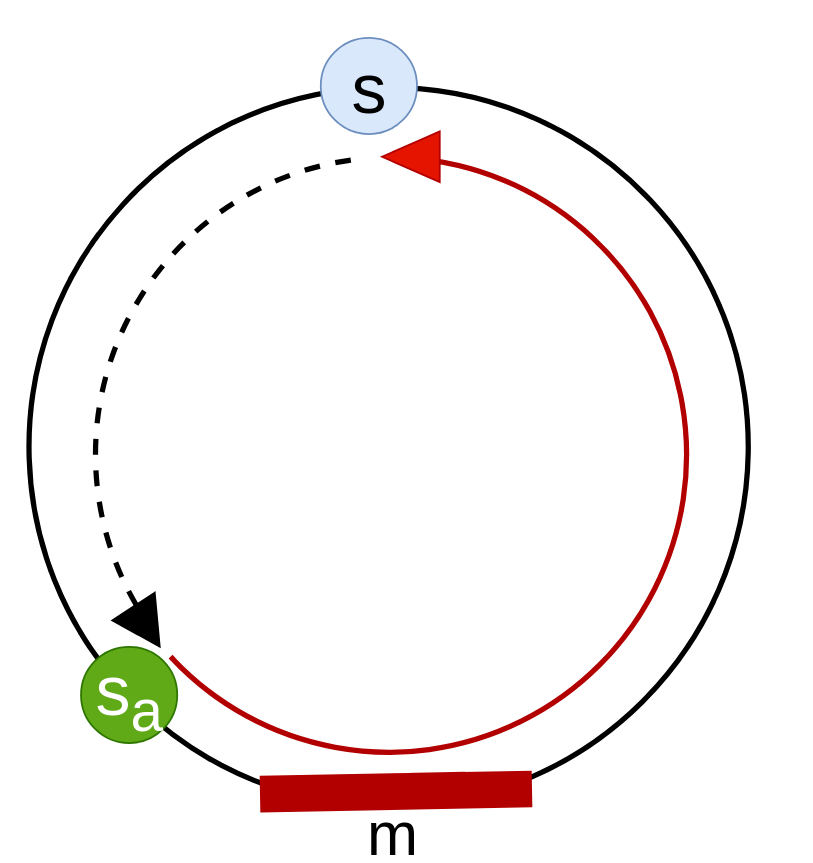} 
        \caption{The vertex $s_a$ is on the left side of $m$ and the journey from $s$ to $s_a$ goes counter-clockwise from $s$. If we are in the case $2.3$, such an exploration exist.}
        \label{fig:2_3a}
    \end{minipage}\hfill
\end{figure}

Recall that in the case $1.3$ the result of algorithm $2$ is returned and in the case $1.3$ there is a journey in $n-1$ time steps going counter-clockwise from $s$ to the vertex $s_h$ on the right side of $m$ ($s_h$ is either the right extremity of $m$ or on the right side of the extremity). Also this journey visits all the vertices between $s$ and the right extremity of $m$ (by going counter-clockwise) in less than $n-1$ time steps, and by assumption $s_a$ is one of these vertices because it is on the left side of $m$ or the right extremity of $m$.

So by assumption there is a counter-clockwise journey from $s$ to $s_a$ in at most $n-1$ time steps. And after arriving at $s_a$ in the counter-clockwise direction, the exploration without blocking also takes place in the counter-clockwise direction until the exploration is complete. Therefore the agent traverses a total of $n-1$ edges for the exploration as shown in figure \ref{fig:2_3a}. Our algorithm will then return this solution that traverses $n-1<\lfloor \frac{3}{2}(n-1) \rfloor$ edges in $2n-2$ time steps.

So our algorithm always returns an exploration of $\mathcal{C}$ that takes at most $2n-2$ time steps and traverses at most $\lfloor \frac{3}{2}(n-1) \rfloor$ edges.
\end{proof}
\begin{lemma}
    For every integer $n\geq 3$, there exists a temporal cycle $\mathcal{C}$ with $L\geq 2n-2$ such that the number of edges to traverse for every exploration is at least $\lfloor \frac{3}{2}(n-1) \rfloor$.
    \label{lemma:l2_cycle}
\end{lemma}
\begin{proof}

We construct the following temporal graph: every snapshot of $\mathcal{C}$ is identical and only the edge $m$ (such as defined in the beginning of the part \ref{subsec:cycle}) is constantly absent.\\
We see that the shortest exploration of $\mathcal{C}$ traverses
$\lfloor \frac{n-1}{2} \rfloor+n-1 = \lfloor \frac{3}{2}(n-1) \rfloor$ edges in the same number of time steps by first traversing the edges in the counter-clockwise direction from $s$ to the left extremity of $m$ ($\lfloor \frac{n-1}{2} \rfloor$ edges) and then traversing $n-1$ edges in the clockwise direction to the other extremity of $m$.
\end{proof}
\subsubsection{The particular case of an exploration in $2n-3$ time steps}
\begin{lemma}
    For every integer $n\geq 3$, there exists a temporal cycle $\mathcal{C}$ with $L=2n-3$ such that there is only one possible exploration and this exploration traverses $2n-3$ edges.
    \label{lemma:l3_cycle}
\end{lemma}
\begin{proof}

We construct the temporal cycle $\mathcal{C}$ with a lifetime $L=2n-3$ as shown in figure \ref{fig:borne_inf_2}. We have $v_1$ the vertex adjacent to $s$ in the clockwise direction and $v_2$ the vertex adjacent to $v_1$ in the clockwise direction. The edge $\{s,v_1\}$ is absent during the time interval $[1,n-2]$, and the edge $\{v_1,v_2\}$ is absent during the interval from time step $[n-1,L]$.
We prove that there is only one possible journey exploring this cycle in $2n-3$ time steps and that it traverses $2n-3$ edges.
\begin{figure}[!h]
\centering
\includegraphics[scale=0.21]{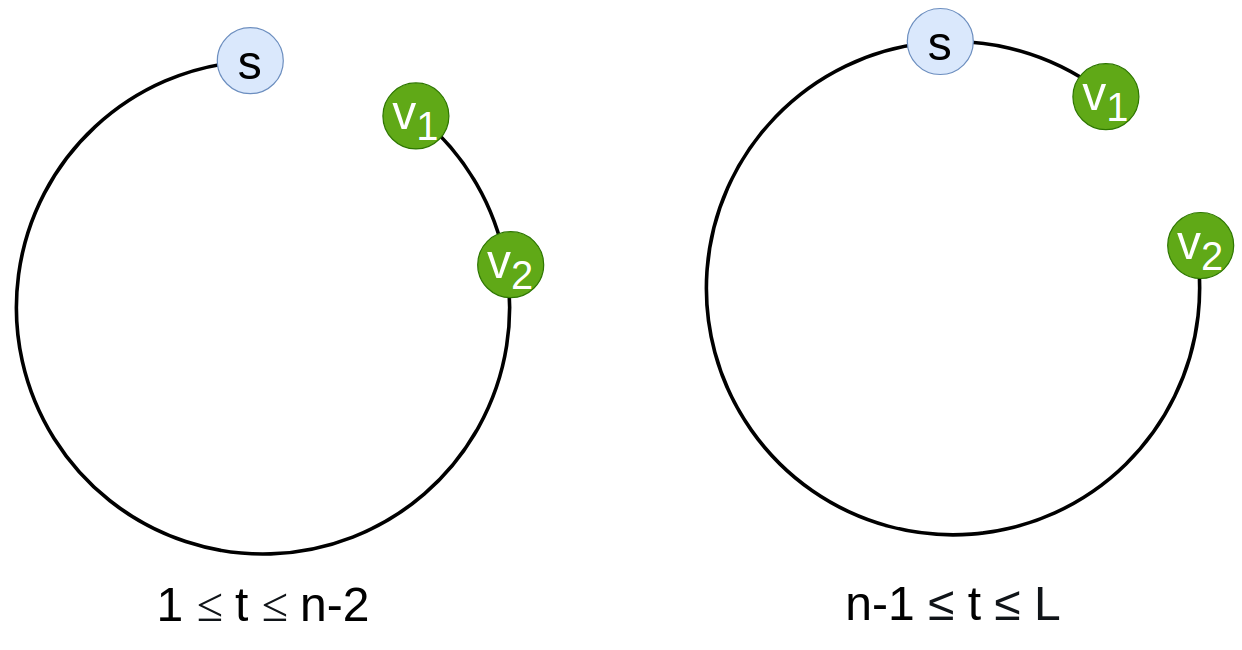}
\caption{The number of edges to traverse to explore this graph is $2n-3$ if $L=2n-2$.}
\label{fig:borne_inf_2}
\end{figure}
These are the two possible cases :
\begin{itemize}
    \item \textbf{The agent visits $v_1$ before $v_2$ :} The agent traverses the edge $\{s,v_1\}$ at the time step $n-1$ and is on $v_1$ at $n$. The agent must then change direction, as the edge $\{v_1,v_2\}$ is absent from time step $n-1$ to $L$, and visits the rest of the vertices by moving in the counter-clockwise direction from $v_1$ during $n-1$ time steps. The agent will have traversed the last edge of its exploration at $2n-2$, by traversing $n$ edges. This solution is invalid because we have to visit the temporal graph before $L$ (i.e in $2n-3$ time steps).
    
    \item \textbf{The agent visits $v_2$ before $v_1$ :} The agent starts by moving counter-clockwise until it reaches $v_2$ at time step $n-1$. However, the edge $\{v_1,v_2\}$ is absent from time step $n-1$ to $L$, so the agent has to go clockwise and complete the exploration by reaching $v_1$ in $n-1$ time steps from vertex $v_2$ traversing the last edge $\{s,v_1\}$ at the time step $2n-3$ because $n-1+n-2=2n-3=L$.
    The agent will have completed the exploration after $2n-3$ time steps by traversing $2n-3$ edges, so this is the only valid solution.
\end{itemize}
We have proven that the only possible exploration of $\mathcal{C}$ in $2n-3$ time steps must traverse $2n-3$ edges, i.e. we have a lower bound of $2n-3$ for the number of edges to traverse for the case where $L=2n-3$.
\end{proof}
Note that the temporal graph presented in the proof of Lemma \ref{lemma:l3_cycle} has already been presented by Ilcinkas et al.\cite{ilcinkas_ring} for the study of ETEXP in the cycle.\\
An overview of the results described by lemmas \ref{lemma:l1_cycle},\ref{lemma:l2_cycle} and \ref{lemma:l3_cycle} are presented in figure \ref{fig:courbe1}. Since Ilcinkas et al.\cite{ilcinkas_ring} proved that there is always an exploration if $L\geq 2n-3$ time steps, we start the value of the lifetime at $2n-3$ in figure \ref{fig:courbe1}. So the bounds are tight for the number of edges traversed in STEXP if $L\geq 2n-3$.
\begin{figure}[!h]
\centering
\includegraphics[scale=0.24]{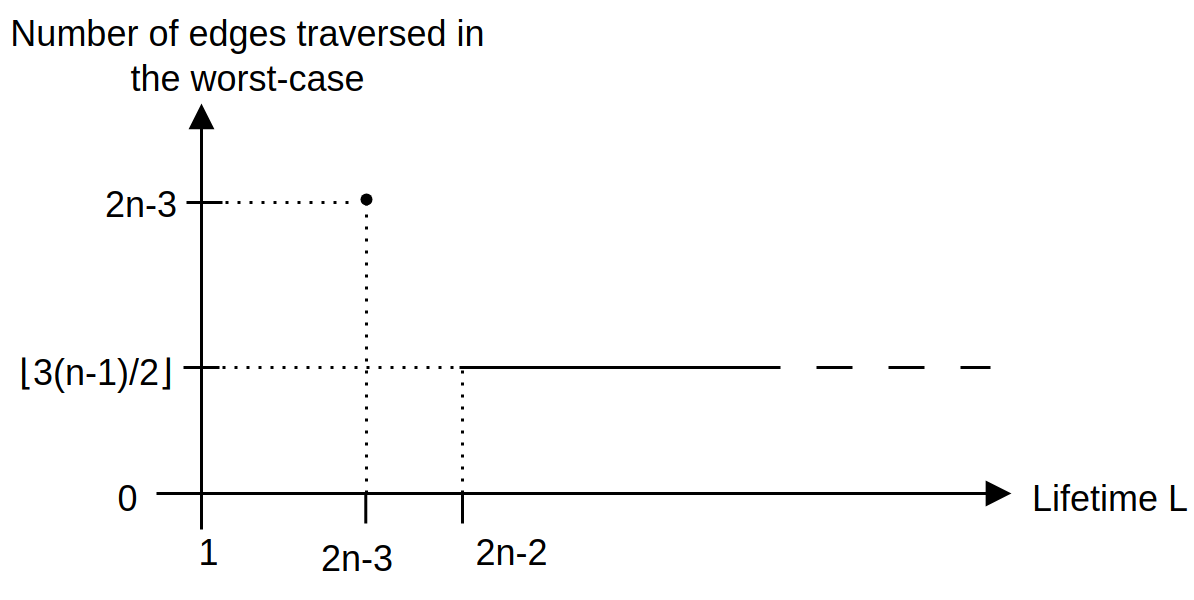}
\caption{A representation of the number of edges traversed in the worst case as a function of the lifetime $L$ for the cycle.}
\label{fig:courbe1}
\end{figure}
\section{Conclusion}
\label{sec:conclusion}
In this paper we have proven that, for a constantly connected temporal graph, there is an exploration in $O(n^{3.5})$ time steps traversing $O(n^{1.5})$ edges. Morever we have presented results on more restricted classes of constantly connected temporal graphs by proving that we have an exploration in $O(kn^2)$ time steps traversing $O(kn)$ edges if every snapshot have a diameter bounded by $k$. There exists a family of temporal graphs with snapshots of bounded diameter for which all explorations takes $\Omega(n^2)$ time steps and traverses $\Omega(n)$ edges, and the later result implies that for every temporal graph of bounded diameter, there is an exploration that takes $O(n^2)$ time steps and traverses $O(n)$ edges. The other restricted class studied is the case where the underlying graph is a cycle. We have proven this interesting result : in the STEXP, the number of edges traversed  in the worst case is different if $L=2n-3$ (i.e. the exploration takes at most $2n-3$ time steps) or $L>2n-3$.
\\Finally, we present the following conjecture : Any constantly connected temporal graph admits an exploration in $\Theta(n^2)$ time steps traversing $\Theta(n)$ edges. Proving that conjecture, or at least showing that the number of traversing edges is $o(n^2)$ when there are $\Theta(n^2)$ time steps, is the main perspective of the present paper. 

\bibliographystyle{plain} 
\bibliography{lipics-v2021-sample-article.bib} 

\newpage
\appendix

\section{Proof of lemmas}
\subsection{Lemma \ref{lemma:l4}}
\label{proof:p1}
\begin{lemma*}
   Let $\mathcal{G}$ be a temporal graph of $n$ vertices and $k$ such that $1\leq k< n$, the lifetime of the graph is $L\geq kn$. \\ Let $u,v$ be two distinct vertices of $V$, such that there exist $kn$ distinct paths (i.e on distinct snapshots) with a number of edges less than or equal to $k$ connecting $u$ and $v$, each of these paths $u\leftrightarrow_k v$ is associated with exactly one time step. Then there exists a journey $u \rightsquigarrow_k v$ in $\mathcal{G}$ using only these time steps.
\end{lemma*}
\begin{proof}
     In this proof, only the time steps where there is a path $u \leftrightarrow_k v$ in $\mathcal{G}$ are considered.
    Note that the case where $k=1$ is trivial, as there must be a single $1$-edge path between the vertices $u$ and $v$ to obtain a journey $u \rightsquigarrow_1 v$, and $1<1*n$ (as $u$ and $v$ are two distinct vertices, so $|V|\geq 2$). In the rest of the proof, we assume that $k>1$.
    To each vertex $w$ a label $l_{w}^t$ is associated, which is the minimal length of a journey from $u$ to $w$ whose last edge is traversed at $t'\leq t$; when there is no possible confusion, we denote this label $l_w$. We have $l_{u}^1=0$ and $\forall{w\neq u}, l_{w}^1=\infty$.
    
    We consider the snapshot at the time step $t_c, 1<t_c<kn$: by assumption it contains \emph{a path} $u \leftrightarrow_k v$. We prove that the snapshot is in one of the following three cases.
      \begin{itemize}
        \item{Case 1:} There exists a journey $u\rightsquigarrow_k v$ where the last edge of the journey is traversed at most in $t_c$. That is to say that $v$ has a label $l_{v}^{t_c}\leq k$.
        \item{Case 2:} There exists a vertex $w\in V$ such that we have $l_w^{t_c-1}=x_1>k$ and $l_w^{t_c}=x_2\leq k$.
        \item{Case 3:} There exists a vertex $w\in V$ such that we have $l_w^{t_c-1}=x_1\leq k$ and $l_w^{t_c}\leq x_1-1$.
    \end{itemize}
     If the snapshot is in case $1$, then there is a journey $u\rightsquigarrow_k v$, so we assume in the rest of the proof that there is no journey $u\rightsquigarrow_k v$ at the time step $t_c$, i.e. $v$ has a label $l_{v}^{t_c}> k$.

    Considering the vertices on the path $u\leftrightarrow_k v$ at $t_c$, these vertices are partitioned into two sets $A$ and $B$. We assign each vertex $w$ of the path to a label $d_w^{t_c}$, which is the distance to $u$ in the path (so $u$ has the labels $l_{u}^{t}=0$ and $d_{u}^{t}=0$ at each time step $t$). A vertex $w$ of the path is in the set $A$ at $t_c$ if $l_w^{t_c-1}\leq d_w^{t_c}$, otherwise $w$ is in the set $B$.

\begin{figure}[!h]
    \centering
    \includegraphics[scale=0.22]{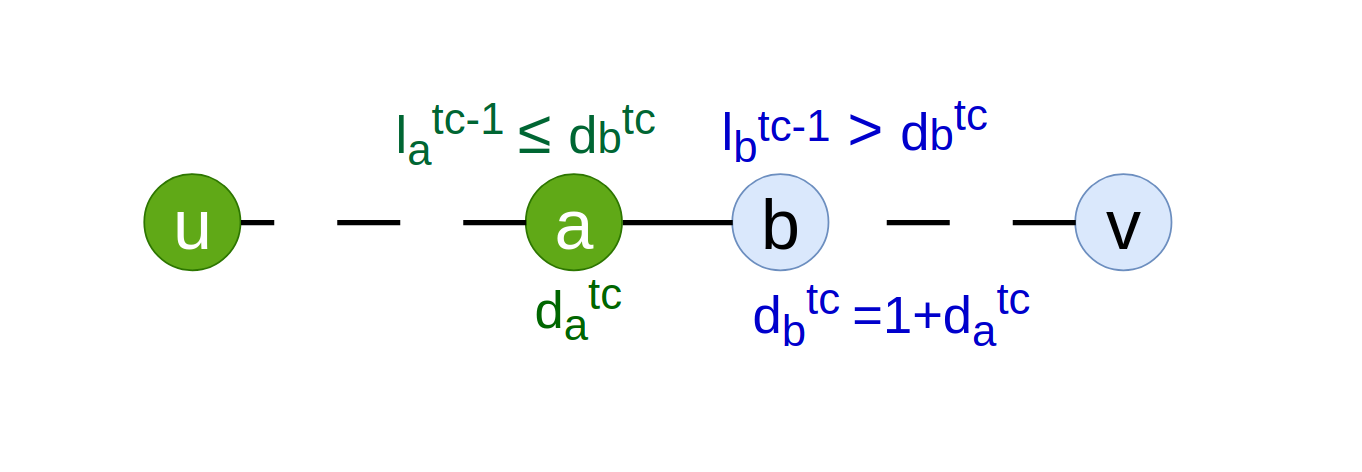}
    \caption{A representation of the path of at most $k$ edges at the time step $t_c$ in the case 2 or 3: vertices in A are in green and vertices in B are in blue.}
    \label{fig:img1}
\end{figure}

By assumption, we are not in the case 1 at time step $t_c>1$ so $l_v^{t_c}> k$ which means $l_v^{t_c-1}> k$ (the length of the shortest journey from $u$ to $v$ can only decrease with time). 
Then at $t_c$, we have $v\in B$ because $d_{v}^{t_c}\leq k< l_{v}^{t_c-1}$ so $d_{v}^{t_c}< l_{v}^{t_c-1}$. Moreover we have $u\in A$ because $l_u^{t_c-1}=0$ and $d_u^{t_c}=0$, so $d_u^{t_c}\geq l_u^{t_c-1}$.

We follow the path $u\leftrightarrow_k v$ of the time step $t_c$ from $u$, vertex by vertex, until we find the first vertex of set $A$, denoted $a$, which is connected by an edge to a vertex of set $B$, denoted $b$, as shown in figure \ref{fig:img1}. The vertices $a$ and $b$ must exist, as we have seen that $u\in A$ and $v\in B$ at $t_c$. So we have $d_b^{t_c} = d_a^{t_c} +1$, and $l_b^{t_c-1}> d_b$ because $b\in B$ and $l_a^{t_c-1}\leq d_a$ because $a\in A$, i.e. $l_b^{t_c-1}>d_b=d_a+1\geq l_a^{t_c-1}+1$, so $l_b^{t_c-1}>l_a^{t_c-1}+1$. Therefore we associate $b$ to the label $l_b^{t_c}=l_a^{t_c-1}+1$, this corresponds to the journey from $u$ to $a$ of $l_a$ edges at the time step $t_c-1$, to which we add the edge $(a,b)$ traversed at the time step $t_c$.
\newline
We are in one of the following two cases:
\begin{itemize}
    \item Either $l_b^{t_c-1}>k$, and we have the label $l_b^{t_c}=l_a^{t_c-1}+1\leq d_a^{t_c}+1 \leq k$ because $d_a^{t_c}\leq k-1$ according to its position in the path $u\leftrightarrow_k v$. So we are in case $2$ because the labels of $b$ are $l_b^{t_c-1}>k$ and $l_b^{t_c}\leq k$.
    \item  Either $l_b^{t_c-1}\leq k$, and we know that $l_b^{t_c-1}>l_a^{t_c-1}+1=l_b^{t_c}$, so we have $l_b^{t_c}<l_b^{t_c-1}$, i.e. $l_b^{t_c}\leq l_b^{t_c-1}-1$. Then we are in the case 3, since we have $l_b^{t_c-1}=x_1\leq k$ and $l_b^{t_c}\leq x_1-1$.
\end{itemize}
    
So as long as we're not in the case 1 (i.e. there's no journey $u\rightsquigarrow_k v$ at $t_c$), then we're necessarily either in case 2 or case 3. However, we can be in case 2 at most $n-1$ times, because there are exactly $n-1$ vertices different from $u$, and a vertex $w\neq u$ cannot have more than one time step $t>1$ such that $l_w^{t-1}>k$ and $l_w^t\leq k$, because the label $l_w$ does not increase with time. And we can be in case 3 at most $(k-1)(n-1)$ times because if the label of the vertex $w\neq u$ at time step $t$ is $l_w^t\leq k$, since we have $l_w^t\geq 1$, $l_w$ can't decrease by at least $1$ more than $\lfloor k-1\rfloor\leq k-1$ times. So Case 3 can only occur at most $(n-1)(k-1)$ times. In total, there can only be $(k-1)(n-1) + (n-1)= kn-k < kn-1$ cases 2 or 3, and by assumption we have $kn$ distinct paths $u\leftrightarrow_k v$ in $\mathcal{G}$. So there is at least one time step $t$ where a path $u\leftrightarrow_k v$ exists in which we are in case 1, i.e. there is a journey $u\rightsquigarrow_k v$ where the last edge traversed in the path is traversed at this time step $t$.
So if there are $kn$ distinct paths with at most $k$ edges connecting $u$ and $v$ in $\mathcal{G}$ with $L\geq kn$, then there exists a journey $u \rightsquigarrow_k v$ (and $v \rightsquigarrow_k u$).
\end{proof}
\subsection{Lemma \ref{lemma:l5}}
\label{proof:p2}
\begin{lemma*}
    Let $G=(V,E)$ be a connected static graph of $n$ vertices and $k,\mbox{ such that } 1< k<n$. Let $Z$ be a subset of vertices of maximal size such that $\nexists{(u,v)\in Z},u \leftrightarrow_k v$, then we have $|Z|<2n/k$.
\end{lemma*}
\begin{proof}
    Let $Z$ be a subset of vertices of maximum size such that $\nexists{(u,u')\in Z},u \leftrightarrow_k u'$. We denote the set $Z=\{z_1,z_2,...\}$ and for each vertex $z_i\in Z$ we associate a subset of vertices $Y_i$ which we define as follows: $\forall{u\in V}, u \leftrightarrow_{k/2} z_i\Leftrightarrow u\in Y_i$.
    Note that if $|Z|\geq 2$, for any distinct pair of vertices $(z_i,z_j)\in Z^2$, we have $Y_i\bigcap Y_j=\emptyset$ since otherwise we would have a path $z_i \leftrightarrow_k z_j$.
    Moreover, we have $\forall{i\in[1,|Z|]},|Y_i|\geq \lfloor k/2 \rfloor +1>k/2$ because the graph $G$ is connected. Let $Y=\bigcup_{i=1}^{|Z|}Y_i$, we have $|Y|\leq n$ and $|Y|=\sum_{i=1}^{|Z|}|Y_i|>|Z|k/2$, so $n>|Z|k/2\implies|Z|<2n/k$.
    So if there is a subset $Z$ such that $\nexists{(u,v)\in Z},u \leftrightarrow_k v$, we have $|Z|<2n/k$.
\end{proof}
\subsection{Lemma \ref{lemma:l6}}
\label{proof:p3}
\begin{lemma*}
Let $\mathcal{G}$ be a temporal graph, let $q\in\mathbb{N}^*$, let $S$ be a subset of vertices such that $|S|\leq n/q$ and let $k=2\sqrt{nq}$.
 If the lifetime $L$ verifies $L\geq 4\frac{n^{2.5}}{\sqrt{q}}$ then for every set $X_k^S$ we have $|X_k^S|< \sqrt{n/q}$.
\end{lemma*}
\begin{proof}
    We note that $2n/k=\sqrt{n/q}$, in the following proof we prove that $|X_k^S|< 2n/k$.

     If $|S|<2n/k$, then we have $|X_k^S|<2n/k$ for every set $X_k^S$, so in the rest of the proof we assume that $|S|\geq 2n/k$.
     We prove by contradiction that for every set $X_k^S$ we have $|X_k^S|<2n/k$ if $|S|\geq 2n/k$. We assume that $|S|\geq 2n/k$ and there is a set $X_k^S$ such that $|X_k^S|\geq 2n/k$ in $\mathcal{G}$, we prove that there exists a pair of vertices $(u,v) \in X_k^S$ such that we have $u\rightsquigarrow_k v$ and $v \rightsquigarrow_k u$, which is a contradiction by the definition of $X_k^S$.
     By the assumption that $|X_k^S|\geq 2n/k$, we are in one of these two cases:\\
     \begin{itemize}
     
     \item{Case 1: $2n/k\leq|X_k^S|< 4n/k$.}\\ According to Lemma \ref{lemma:l5}, if $|X_k^S|\geq 2n/k$, then at each time step $t$, there exists a subset of vertices $Z_t\subset X_k^S$ with $|Z_t|<2n/k$ such that $\nexists{(u,v)\in Z_t^2}$ such as $u \leftrightarrow_k v$ and $Z_t$ has a maximal size. We deduce that $\exists v\in X_k^S\backslash Z_t,\exists u\in Z_t$ such as $u \leftrightarrow_k v$ in $G_t$ . There are $\frac{|X_k^S|(|X_k^S|-1)}{2}$ distinct pairs of $X_k^S$ vertices, so according to the pigeon-hole principle, after $\frac{|X_k^S|^2}{2}kn$ time steps there exists a pair of vertices $(u,v) \in X_k^S$ such that there are $kn$ distinct $u \leftrightarrow_k v$ paths in $\mathcal{G}$. Then by Lemma \ref{lemma:l4} we have a pair of vertices $(u,v) \in X_k^S$ such that we have $u\rightsquigarrow_k v$ and $v \rightsquigarrow_k u$ in $\mathcal{G}$. By assumption we have $|X_k^S|<4n/k$, so $\frac{|X_k^S|^2}{2}kn<\frac{8n^2}{k^2}kn=8\frac{n^3}{k}=4\frac{n^{2.5}}{\sqrt{q}}=L$. So, according to the definition of $X_k^S$, there is a contradiction because we have a pair of vertices $(u,v) \in X_k^S$ such that $u\rightsquigarrow_k v$ and $v \rightsquigarrow_k u$ in $\mathcal{G}$.\\

     \item{Case 2: $|X_k^S|\geq 4n/k$.}\\ By Lemma \ref{lemma:l5}, at each time step $t$, there exists a subset of vertices $Z_t\subset X_k^S$ with $|Z_t|<2n/k$ such that $\forall{v\in S\backslash Z_t},\exists{u\in Z_t},u\leftrightarrow_k v$. So at each time step $t$, there exists a subset of vertices of $Y_t = X_k^S\backslash Z_t$ such that $\forall{u\in Y_t}\exists{v\in Z_t}, u \leftrightarrow_k v$, we note that we have $|Y_t|>|X_k^S|-2n/k$. According to the pigeon-hole principle, after $\frac{|X_k^S|}{|X_k^S|-2n/k}|X_k^S|kn$ time steps, there exists a pair of vertices $(u,v)\in X_k^S$ such that there are $kn$ distinct paths $u \leftrightarrow_k v$ in $\mathcal{G}$. Then by Lemma \ref{lemma:l4} we have $u\rightsquigarrow_k v$ and $v \rightsquigarrow_k u$. By assumption, $|X_k^S|\geq 4n/k$ and let $i=\frac{|X_k^S|}{2n/k}\geq 2$, so we have $\frac{|X_k^S|}{|X_k^S|-2n/k}=i/(i-1)$, as $i\geq 2$ we have $\frac{|X_k^S|}{|X_k^S|-2n/k}=i/(i-1)\leq 2$.
     The total number of time steps is : $\frac{|X_k^S|}{|X_k^S|-2n/k}|X_k^S|kn<2|X_k^S|kn\leq 2|S|kn\leq 4\frac{n^{2.5}}{\sqrt{q}}=L$. Therefore we have a contradiction because there is a pair of vertices $(u,v) \in X_k^S$ such that $u\rightsquigarrow_k v$ and $v \rightsquigarrow_k u$ in $\mathcal{G}$.
     \end{itemize}
     We have proven that we have a contradiction if there is a set $X_k^S$ such that $|X_k^S|\geq 2n/k$ in $\mathcal{G}$. So $|X_k^S|<2n/k=\sqrt{n/q}$ in $\mathcal{G}$ if $L\geq 4\frac{n^{2.5}}{\sqrt{q}}$.
\end{proof}
\subsection{Lemma \ref{lemma:l7}}
\label{proof:p4}
\begin{lemma*}
    Let $\mathcal{G}$ be a temporal graph, let $q\in\mathbb{N}^*$ and $S$ a subset of vertices such that $|S|\leq n/q$.
    If $L\geq 5\frac{n^3}{q}$, then there exists a journey in $\mathcal{G}$ visiting $|S|\sqrt{q/n}$ vertices of $S$ by traversing at most $2n$ edges.
\end{lemma*}
\begin{proof}
In this proof, we will use the notation 
$k=2\sqrt{nq}$. First we present the following property :
let $t_1$ and $t_2$ be two timesteps that verify $t_1+4\frac{n^{2.5}}{\sqrt{q}}\leq t_2\leq L$ and $\mathcal{G}'=\mathcal{G}_{[t_1,t_2]}$, for every set $X_k^S$ in $\mathcal{G}'$, we have $\forall{v\in S\backslash X_k^S},\exists{u\in X_k^S} \mbox{ such that } u\rightsquigarrow_k v \mbox{ and } v\rightsquigarrow_k u$. In the rest of the proof, when we talk of a set $X_k^S$, it refers to the set $X_k^S$ of minimum size. This property is represented in figure \ref{fig:imgsupp}. 
 \\To define the algorithm that constructs the journey in $\mathcal{G}$ that visits $ |S|\sqrt{q/n}$ vertices of $S$, we first partition the set $S$ into $l=\lfloor |S|\sqrt{q/n} \rfloor+1$ subsets denoted $S_i,1\leq i\leq l$ described below. We note $C=4\lfloor\frac{n^{2.5}}{\sqrt{q}}\rfloor$.
    $$
    S_i = \left\{
        \begin{array}{ll}
            X_k^{S\backslash S'} \mbox{with } S'=\bigcup_{j=1}^{i-1}S_j \mbox{ in } \mathcal{G}_{[(i-1)C,iC]}& \mbox{ if } 1\leq i<l \\
            S\backslash \bigcup_{j=1}^{i-1}S_j & \mbox{if } i=l
        \end{array}
    \right.$$

We prove that $\forall_{i<l}, 0<|S_i|<\sqrt{n/q}$. We recall that $l=\lfloor |S|\sqrt{q/n}\rfloor +1$. By the definition of $S_i$ above, if $i<l$ then $S_i = X_k^{S\backslash S'}$ with $S'=\bigcup_{j=1}^{i-1}S_j$. And we know from Lemma \ref{lemma:l6} that for any set of vertices $A,\mbox{ such that }|A|\leq n/q$ in a temporal graph of lifetime greater than $C=4\lceil\frac{n^{2. 5}}{\sqrt{q}}\rceil$ we have, $|X_k^A|<\sqrt{n/q}$ so $|S_i|<\sqrt{n/q}$ in $\mathcal{G}_{[(i-1)C,iC]}$. We then show that $|S_i|>0$, we have $|S'|= \sum_{j=1}^{i-1}|S_j|$, and we have seen that $\forall{j<l}, |S_j|<\sqrt{n/q}$ so $|S'|< \sum_{j=1}^{i-1}\sqrt{n/q} = (i-1)\sqrt{n/q} <|S|\sqrt{q/n}\sqrt{n/q}= |S|$, i.e. $|S|>|S'|$. So $|S\backslash S'|=|S|-|S'|>0$, so we have $|S\backslash S'|>0$, i.e. $|S_i| = |X_k^{S\backslash S'}|>0$ because $|X_k^{S\backslash S'}|>0$ if $S\backslash S'$ is non-empty.

In the case where $i=l$, we prove that $|S_i|>0$. We have seen that $|\bigcup_{j=1}^{l-1}S_j|<(l-1)\sqrt{n/q}<l\sqrt{n/q}=|S|\sqrt{q/n}\sqrt{n/q}=|S|$ i.e. $|\bigcup_{j=1}^{l-1}S_j|<|S|$. Also $|S_l|=|S\backslash \bigcup_{j=1}^{l-1}|S_j|=|S|-|\bigcup_{j=1}^{l-1}|S_j|>0$.\\
So $\forall{i\in[1,l]},0<|S_i|<\sqrt{n/q}$.

Note that, by definition, $\forall{i\in[1,l-1]} \mbox{ and } S''=S\backslash \bigcup_{j=1}^{i-1}S_j$, and the graph $\mathcal{G}_{[(i-1)C,iC]}$, we have $S_i=X_k^{S''} \mbox{ and } \forall{v\in S_{i+1}}, v\in S''\backslash S_i$. So in $\mathcal{G}_{[(i-1)C,iC]}$, we have $\forall{v\in S_{i+1}}, \exists{u\in S_i}$, such that we have the journeys $ u\rightsquigarrow_k v \mbox{ and } v\rightsquigarrow_k u \mbox{ in } \mathcal{G}_{[(i-1)C,iC]}$. 
In the rest of the proof, $\forall{i\in [1,l-1]}$, we denote $u_i(v), v\in S_{i+1}$ the vertex $u_i(v)\in S_i$ such that we have $u_i(v)\rightsquigarrow_k v \mbox{ and } v\rightsquigarrow_k u_i(v) \mbox{ in } \mathcal{G}_{[(i-1)C,iC]}$ (we can also denote this vertex $u_i$ when the context is clear enough).\\
We construct a set of journeys $M$ as follows. First we pose $M=\bigcup_{i=1}^{l-1} M_i$ such that the sets $M_i$ form a partition of $M$. We have $\forall{i\in[1,l-1]}$, $$M_i=\{u_i(v)\rightsquigarrow_k v,\forall{v\in S_{i+1}}\}$$ An example of the subsets of $M_i,1\leq i<l$ are shown in figure \ref{fig:img3}.\\
We then prove that we can concatenate $\lfloor|S|\sqrt{q/n} \rfloor$ journeys of $M$ to obtain a journey in $\mathcal{G}$ visiting $\lfloor |S|\sqrt{q/n}\rfloor +1$ vertices of $S$ by selecting exactly one journey in each subset $M_i,1\leq i<l$.\\
We know that $|S_l|>0$, let $u_l\in S_l$, and we have seen that in $\mathcal{G}_{[(l-1)C,lC]}$, we have a vertex $u_{l-1}(u_l)\in S_{l-1}$ such that we have the journey belonging to the set $M_{l-1}$ verifying $u_{l-1}(u_l)\rightsquigarrow_k u_l$. We note $u_{l-1}$ the vertex $u_{l-1}(u_l)$, and we select this journey to be part of the final journey. In the same way, in $\mathcal{G}_{[(l-2)C,(l-1)C]}$ we select the journey belonging to $M_{l-2}$ verifying $u_{l-2}(u_{l-1})\rightsquigarrow_k u_{l-1}$. We iterate this operation until we select the journey $u_1(u_2)\rightsquigarrow_k u_2$ belonging to $M_1$ in $\mathcal{G}_{[1,C]}$.
\\ We then concatenate each of the $l-1 = \lfloor|S|\sqrt{q/n} \rfloor$ selected journeys, giving us a journey $u_1\rightsquigarrow_k u_2 \rightsquigarrow_k \mbox{...}u_{l-1}\rightsquigarrow_k u_{l}$. This journey visits $\lfloor|S|\sqrt{q/n} \rfloor+1>|S|\sqrt{q/n}$ vertices from $S$ with $\lfloor|S|\sqrt{q/n} \rfloor*k\leq \frac{n}{q}\sqrt{q/n}2\sqrt{qn}=2n$ edges traversed and in $\lfloor|S|\sqrt{q/n} \rfloor*C\leq 5\frac{n^3}{q}=L$ time steps.
\end{proof}
\subsection{Lemma \ref{lemma:l8}}
\label{proof:p5}
\begin{lemma*}
     Let $\mathcal{G}$ be a temporal graph with a starting vertex $s$, let $q\in\mathbb{N}^*$ with $q<n/4$ and $S$ a subset of vertices such that $|S|= n/q$.
     If the lifetime of $\mathcal{G}$ is $L\geq 6\frac{n^{3.5}}{q^{1.5}}$ then there exists a journey, starting at $s$, visiting $|S|/2$ vertices of $S$ and traversing at most $4\frac{n^{1.5}}{\sqrt{q}}$ edges.  
\end{lemma*}
\begin{proof}
    Given $p=\sqrt{n/q}$ and $k=2\sqrt{nq}$, we first show by induction that $\forall{j\leq p}$, we can construct $j$ successive journeys such that the journey $i\in[1,j]$ visits $p-i+1$ vertices of $S$ that has not already been visited (during the $i-1$ previous journeys) by traversing at most $2n$ edges in $n^3/q$ time steps. \\
    For the initialization, we take $j=1$. Recall that we have $k=2\sqrt{qn}<n $ because $q<n/4$, so by Lemma \ref{lemma:l7} there exists a journey visiting $|S|\sqrt{q/n}=\sqrt{\frac{n}{q}}=p$ vertices of $S$ by traversing at most $ 2n$ edges in $5\frac{n^3}{q}$ time steps. So for $j=1$, we have a journey that visits $p=p-j+1$ vertices of $S$ in $2n$ edges and $n^3/q$ time steps.\\
    The induction hypothesis is that there exists an integer $j\in[1,p[$ such that there exists $j$ successive journeys such that the journey $i\in [1,j]$ visits $p-i+1$ vertices of $S$ in $2n$ edges and $5n^3/q$ time steps. We prove that $j+1$ similar successive journeys can be constructed.\\
    By the induction hypothesis, we know that there are $j$ successive journeys that visit a total of $\sum_{i=1}^j p-i+1< jp$ vertices of $S$, so if $S'$ is the set of vertices of $S$ not already visited, we have $|S'|>|S|-jp$. By Lemma \ref{lemma:l7}, in $S'$ with $|S'|\leq |S|\leq n/q$, there exists a journey visiting $|S'|\sqrt{q/n}$ vertices of $S'$ by traversing at most $ 2n$ edges in $5\frac{n^3}{q}$ time steps.
    Now we have $|S'|\sqrt{q/n}>(|S|-jp)\sqrt{q/n}=(n/q-j*\sqrt{n/q})\sqrt{q/n}=\sqrt{\frac{n}{q}}-j=p-j$ vertices of $S'$, and the number of edges traversed is $2n$ and the number of time steps is $n^3/q$ so we have proven the induction.
    So, $\forall{j\leq p}$, there exist $j$ successive journeys such that the journey $i,1\leq i\leq j$ visits $p-i+1$ vertices of $S$ in $2n$ edges and $5\frac{n^3}{q}$ time steps.\\

    Recall that, by Lemma \ref{lemma:l1}, for any pair of vertices $(a,b)$ in $\mathcal{G}$, there exists a journey from $a$ to $b$ in $n-1$ time steps (and traversing at most $n-1$ edges).
    We can therefore create $p=\sqrt{n/q}$ successive journeys as described above (each traversing at most $2n$ edges), which we call journeys of type $1$, and connect them by pairs with journeys of $n-1$ time steps and $n-1$ edges which we call the journeys of type $2$. The total number of vertices of $S$ visited by the journeys of type $1$ is $p+(p-1)+(p-2)+..+1 = p(p-1)/2=p^2/2 - p/2= \frac{n}{2q}-\frac{\sqrt{n}}{2\sqrt{q}}$. And the number of edges traversed by the journeys of type $1$ and $2$ is at most $p(2n + n-1)<3pn = 3\frac{n^{1.5}}{\sqrt{q}}$. In the same way, we compute the total number of time steps of the $p$ successive journeys of type $1$ and $2$: $p(5n^3/q + n-1) < n^{1.5}/\sqrt{q} + 5n^{3.5}/q^{1.5}$.

    After traversing the $p$ successive journeys of type $1$, we have visited at least $\frac{n}{2q}-\frac{\sqrt{n}}{2\sqrt{q}}$ vertices of $S$, we then visit $\frac{\sqrt{n}}{2\sqrt{q}}$ vertices of $S$ not already visited by making as many successive journeys of type $2$ (each journey visiting exactly one of the $\frac{\sqrt{n}}{2\sqrt{q}}$ vertices of $S$ not already visited). Since a journey of type $2$ traverses at most $n-1$ edges in $n-1$ time steps, if we concatenate all the journeys of type $1$ and $2$, in total we have a final journey visiting $n/2q=|S|/2$ vertices of $S$ traversing at most $3\frac{n^{1. 5}}{\sqrt{q}} + \frac{\sqrt{n}}{2\sqrt{q}}(n-1)<4\frac{n^{1.5}}{\sqrt{q}}$ edges and in $n^{1.5}/\sqrt{q} + 5n^{3.5}/q^{1. 5} + \frac{\sqrt{n}}{2\sqrt{q}}(n-1)<\frac{3n^{1.5}}{2\sqrt{q}} + 5\frac{n^{3.5} }{q^{1.5}}<6\frac{n^{3.5} }{q^{1.5}}=L$ time steps.
\end{proof}
\end{document}